\documentclass{amsart}

\usepackage[latin1]{inputenc}
\usepackage{color}
\usepackage{url}
\usepackage{pstricks}
\usepackage{multicol}
\usepackage{tikz}
\usetikzlibrary{calc,through,backgrounds}

\usepackage[enableskew]{youngtab}
\usepackage{young}

\renewcommand{\blue}{\textcolor{blue}}

\newcommand{\threep}{\mbox{$3'$}}

\newcommand{\ninep}{\mbox{$9'$}}

\newcommand{\ten}{$10$}

\makeatletter
  \newcommand\tinyv{\@setfontsize\tinyv{1pt}{1}}
\makeatother

\newcommand{\ip}{\small\mbox{{$i$\!\!} {\tiny $\stackrel{+}{}$} {\small\!\!\!$1$}}}
\newcommand{\im}{\small\mbox{{$i$\!\!\!} {\tiny $\stackrel{-}{}$} {\small\!\!\!$1$}}}
\newcommand{\jp}{\small\mbox{{$j$\!\!\!} {\tiny $\stackrel{+}{}$} {\small\!\!\!$1$}}}
\newcommand{\jm}{\small\mbox{{$j$\!\!\!} {\tiny $\stackrel{-}{}$} {\small\!\!\!$1$}}}

\newcommand{\ycenter}{\Yvcentermath1}

\newcommand{\mread}{\operatorname{read}}

\def\1#1{\operatorname{\bf 1}_k}

\def\F{\mathcal{F}}

\def\jdt{\operatorname{jdt}}
\def\sh{\operatorname{sh}}

\def\ZZ{{\mathbb Z}}

\def\QQ{{\mathbb Q}}

\newtheorem{theorem}{Theorem}[section]

\newtheorem{lemma}[theorem]{Lemma}
\newtheorem{proposition}[theorem]{Proposition}

\newtheorem{definition}[theorem]{Definition}

\theoremstyle{definition}
\newtheorem{example}[theorem]{Example}
\newtheorem{remark}[theorem]{Remark}	

\title[Staircase skew Schur functions are Schur $P$-positive]
{Staircase skew Schur functions \\ are Schur $P$-positive}

\author{Federico Ardila}
\address{Department of Mathematics, San Francisco State University.}
\email{federico@math.sfsu.edu}
\urladdr{http://math.sfsu.edu/federico/}

\author{Luis G. Serrano}
\address{LaCIM, Universit\'e du Qu\'ebec \`a Montr\'eal.}
\email{serrano@lacim.ca}
\urladdr{http://www.thales.math.uqam.ca/~serrano/}

\date{\today}
\thanks{
The first author was partially supported by the National Science Foundation CAREER Award
DMS-0956178, the National Science Foundation Grant DMS-0801075, and the SFSU-Colombia Combinatorics Initiative. The second author was supported by a National Science and Engineering Council of Canada  (NSERC) PDF Award. 
}

\begin{document}

\begin{abstract}
We prove Stanley's conjecture that, if $\delta_n$ is the staircase shape, then the skew Schur functions $s_{\delta_n / \mu}$ are non-negative sums of Schur $P$-functions. We prove that the coefficients in this sum count certain fillings of shifted shapes. In particular, for the skew Schur function $s_{\delta_n / \delta_{n-2}}$, we discuss connections with Eulerian numbers and alternating permutations.
\end{abstract}

\maketitle

\section{introduction}

The Schur functions $s_{\lambda}$, indexed by partitions $\lambda$, form a basis for the ring $\Lambda$ of symmetric functions. These are very important objects in algebraic combinatorics. They play a fundamental role in the study of the representations of the symmetric group and the general linear group, and the cohomology ring of the Grassmannian \cite{Ful97}. The Schur $P$-functions $P_{\lambda}$, indexed by strict partitions, form a basis for an important subring $\Gamma$ of $\Lambda$. They are crucial in the study of the projective representations of the symmetric group, and the cohomology ring of the isotropic Grassmannian \cite{HH92} \cite{Joz91}.

The goal of this paper is to prove the following conjecture of Richard Stanley \cite{Stanconj}: If $\delta_n$ is the staircase shape and $\mu \subset \delta_n$, then the \emph{staircase skew Schur function} $s_{\delta_n / \mu}$, which belongs to the ring $\Gamma$, is a nonnegative sum of Schur $P$-functions. We find a combinatorial interpretation for the coefficients in this expansion in terms of Shimozono's compatible fillings \cite{Shi99}. Furthermore, we discuss connections between the special case of the skew Schur function $s_{\delta_n / \delta_{n-2}}$ and alternating permutations, and show an expansion of these in terms of the elementary symmetric functions.

The paper is organized as follows. In Section \ref{sec:prelim} we recall some basic definitions, including Schur and Schur $P$-functions. In Section \ref{sec:staircase} we discuss the staircase Schur functions and prove that they are indeed in the subring $\Gamma$ of $\Lambda$ generated by the Schur $P$-functions. In Section \ref{sec:thm} we state our main result, Theorem \ref{th:main}, which states that the (non-negative integer) coefficients of $P_{\lambda}$ is the number of ``$\delta_n/\mu$-compatible" fillings of the shifted shape $\lambda$.
In Section \ref{sec:jdt} we prove the key proposition that, in the particular case of staircase skew shapes $\delta_n/\mu$, jeu de taquin respects $\delta_n/\mu$--compatibility. Finally in Section \ref{sec:proof} we prove Theorem \ref{th:main}.

The Schur P-positivity of staircase Schur functions has also been proved independently by Elizabeth Dewitt and will appear in her forthcoming thesis \cite{Dew12}

\noindent \textbf{Acknowledgments.}
We would like to thank Richard Stanley for telling us about his conjecture and about Proposition \ref{prop:e}. \cite{Stanconj} 
We also thank Ira Gessel, Peter Hoffman, Tadeusz J\'ozefiak, Bruce Sagan, and John Stembridge for valuable conversations.

\section{Preliminaries}\label{sec:prelim}

A \emph{partition} is a sequence $\lambda = (\lambda_1, \lambda_2, \ldots, \lambda_l) \in \ZZ^l$ with $\lambda_1 \ge \lambda_2 \ge \cdots \ge \lambda_l > 0$. The \emph{Ferrers diagram}, or \emph{shape} of $\lambda$ is an array of square cells in which the $i$-th row has $\lambda_i$ cells, and is left justified with respect to the top row. 
The \emph{size} of~$\lambda$ is~$| \lambda | := \lambda_1 + \lambda_2 + \cdots + \lambda_l$. We denote the number of rows of $\lambda$ by $\ell(\lambda) := l$.

A \emph{strict partition} is a sequence $\lambda = (\lambda_1, \lambda_2, \ldots, \lambda_l) \in \ZZ^l$ such that $\lambda_1 > \lambda_2 > \cdots > \lambda_l > 0$. The \emph{shifted diagram}, or \emph{shifted shape} of $\lambda$ is an array of square cells in which the $i$-th row has $\lambda_i$ cells, and is shifted $i-1$ units to the right with respect to the top row.

For example, the shape~$(5,3,2)$ and the shifted shape~$(5,3,2)$, of size $10$ and length $3$, are shown below.
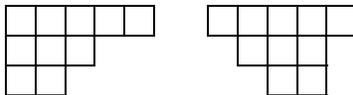
\[
 \young(~~~~~,~~~,~~) \quad \quad \young(~~~~~,:~~~,::~~)
\]
A \emph{skew (shifted) diagram} (or shape) $\lambda / \mu$ is obtained by removing a (shifted) shape $\mu$ from a larger shape $\lambda$ containing $\mu$.

A \emph{semistandard Young tableau} or \emph{SSYT} $T$ of shape $\lambda$ is a filling of a Ferrers shape~$\lambda$ with letters from the alphabet $X = \{1 < 2 < \cdots \}$ which is weakly increasing along the rows and strictly increasing down the columns.

A \emph{shifted semistandard Young tableau} or \emph{shifted SSYT} $T$ of shape $\lambda$ is a filling of a shifted shape~$\lambda$ with letters from the alphabet $X' = \{1' < 1 < 2' < 2 < \cdots \}$ such that:
\begin{itemize}
\item rows and columns of $T$ are weakly increasing;
\item each $k$ appears at most once in every column;
\item each $k'$ appears at most once in every row;
\item there are no primed entries on the main diagonal.
\end{itemize}
If $T$ is a filling of a shape $\lambda$, we write $\sh(T) := \lambda$. 
The \emph{content} of a (shifted) SSYT $T$ is the vector $(a_1, a_2, \ldots)$, where~$a_i$ is the number of times the letters~$i$ and~$i'$ appear in $T$.

A (shifted) SSYT is standard, if it contains the letters $1,2, \ldots, |\lambda|$, each exactly once. In the shifted case, these letters are all unprimed. If that is the case, we call it a (shifted) SYT.  A \emph{skew (shifted) Young tableau} is defined analogously.

\begin{example}
The following are examples of a SSYT and a shifted SSYT, both having shape $\lambda = (5,3,2)$ and content $(2,1,1,2,2,1,0,0,1)$.
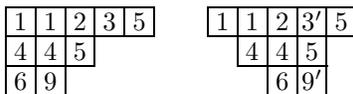
\[ \ycenter
 \young(11235,445,69) \quad  \quad \young(112{\threep}5,:445,::6{\ninep})
\]
\end{example}

In a SYT or a shifted SYT $T$, the pair of entries $(i,j)$, where $i < j$, forms an \emph{ascent} if $j$ is located weakly north and strictly east of $i$. We abbreviate and say that $j$ is \emph{northEast} of $i$. The pair $(i,j)$ forms a \emph{descent} if $j$ is located strictly south and weakly west, or \emph{Southwest}, of $i$. Note that $(i,j)$ could be neither an ascent nor a descent.

When $j=i+1$, the pair $(i,i+1)$ must be either an ascent or a descent, and we abbreviate and call $i$ an ascent or a descent as appropriate.  An entry $i$ forms a \emph{peak} if $i-1$ is an ascent and $i$ is a descent. 

%

\begin{example}\label{ex:ascentdescent}
The figure below shows a SYT of shape $\delta_4:=(4,3,2,1)$ and a shifted SYT of shape $(5,3,2)$, both with descent set $(2,4,5,7,9)$, ascent set $(1,3,6,8)$, and peak set $(2,4,7,9)$.
 \[
  \ycenter \young(1247,359,6{\ten},8) \quad  \quad \young(12479,:358,::6{\ten})
 \]
\end{example}

For a (shifted) Young tableau $T$ with content $(a_1, a_2, \ldots)$, we let 
$x^T = x_1^{a_1} x_2^{a_2} \cdots.$
For each partition $\lambda$, the \emph{Schur function} $s_{\lambda}$ is defined as the generating function for semistandard Young tableaux of shape $\lambda$, namely
\[
s_{\lambda} = s_{\lambda} (x_1, x_2, \ldots) := \sum_{\sh (T) = \lambda} x^T.
\]

It is well known (see e.g., \cite{Sta99}) that the \emph{power sum symmetric functions} $p_i = p_i(x_1,x_2,\ldots) := x_1^i + x_2^i + \cdots$ are a generating set, and the Schur functions $s_{\lambda}$ are a linear basis, for the ring $\Lambda$ of symmetric functions.

For each strict partition $\lambda$, the \emph{Schur $P$-function} $P_{\lambda}$ is defined as the generating function for shifted Young tableaux of shape $\lambda$, namely
\[
P_{\lambda} = P_{\lambda} (x_1, x_2, \ldots) := \sum_{\sh (T) = \lambda} x^T.
\]
%
The Schur $P$-functions form a basis for the subring $\Gamma$ of $\Lambda$ generated by the odd power sums, $\Gamma := \QQ[p_1, p_3, \ldots]$. This ring also has the presentation
\[
\Gamma = \{f \in \Lambda \, : \, f(t,-t,x_1,x_2,\ldots) = f(x_1,x_2,\ldots)\}.
\]
See, e.g., \cite{HH92}. The \emph{skew Schur functions} $s_{\lambda / \mu}$ and the \emph{skew Schur $P$
-functions}~$P_{\lambda / \mu}$ 
are defined similarly for a skew (shifted) shape~$\lambda / \mu$.

\section{The skew Schur functions $s_{{\delta_n}/{\delta_{n-2}}}$ and $s_{{\delta_n}/{\mu}}$.} \label{sec:staircase}

\begin{definition}
The staircase $\delta_n$ is the shape $(n, n - 1,\ldots,2,1)$. Denote $s_{{\delta_n}/{\delta_{n-2}}}=: F_{2n-1}$, and let $\F = \F(x_1, x_2, \ldots) := \sum_{n \geq 1} F_{2n-1}$.
\end{definition}

The symmetric function $F_{2n-1}$ is one of the main subjects of study of this paper. It has nice expansions in terms of the power and elementary symmetric functions.

\begin{definition}
A permutation $a_1a_2\ldots a_n$ of $\{1,\ldots, n\}$ is said to be \emph{alternating} if $a_1<a_2>a_3<a_4>\cdots$.
\end{definition}

\begin{proposition} \label{prop:p} (\cite{Fou76}) Let $E_k$ be the number of alternating permutations of $\{1, \ldots, k\}$, and let $z_{\lambda} := \prod_{i \ge 1} \frac{i^{m_i}}{m_i!}$ for the partition $\lambda = 1^{m_1} 2^{m_2} \cdots$. We have
\[
F_{2n-1} = \sum_{\lambda \in OP(2n-1)} \frac{E_{l(\lambda)}}{z_{\lambda}} p_\lambda
\]
where $OP(2n-1)$ is the set of partitions of $2n-1$ into odd parts.
\end{proposition}

The following proposition expresses the $\F$ in terms of the elementary symmetric functions. Equivalent formulas appear in \cite{Car73}, \cite{Ges77}, \cite[p. 9]{Ges90} and \cite[Corollary 4.2.20]{GJ83}.

\begin{proposition} \label{prop:e} We have
\[
\F = \frac{e_1-e_3+e_5- \cdots}{1-e_2+e_4-\cdots},
\]
where $e_k = \sum_{i_1 < \cdots < i_k} x_{i_1} \cdots x_{i_k}$ is the $k$-th elementary symmetric function.
\end{proposition}

\begin{proof}
Consider a SSYT $T$ of shape ${\delta_n}/{\delta_{n-2}}$ with $n \geq 2$ which contains a $1$. Let the leftmost $1$ occur in the (top entry on the) $k$th column. When we remove that $1$, we are left with a SSYT of  shape ${\delta_k}/{\delta_{k-2}}$ containing no $1$s and a SSYT of shape ${\delta_{n-k}}/{\delta_{n-k-2}}$. It follows that
\[
\F(x_1, x_2, \ldots) - \F(x_2, x_3, \ldots) = x_1 + \F(x_2, x_3, \ldots) x_1\F(x_1, x_2, \ldots).
\]
Denoting $\F_i:= \F(x_i, x_{i+1}, \ldots)$, we rewrite this as $\F_1 = \frac{x_1+\F_2}{1-x_1\F_2}$, which gives that $\arctan \F_1 = \arctan x_1 + \arctan \F_2$ as formal power series. Iterating, we obtain  
\[
\arctan \F_1 = \arctan x_1 + \arctan x_2 + \cdots,
\]
from which the desired formula follows by taking the tangent of both sides.
\end{proof}

More importantly for us, we observe that $\F$ is in the subring $\Gamma$ of $\Lambda$.

\begin{lemma} \cite{Stanconj}
The skew Schur functions $F_{2n-1}$ and, more generally, the \emph{staircase skew Schur functions} $s_{{\delta_n}/{\mu}}$, are in the subring $\Gamma$ of $\Lambda$.
\end{lemma}

\begin{proof} From the equation $e_k(t, -t, x_1, x_2, \ldots) = e_k - t^2e_{k-2}$ and Proposition \ref{prop:e}  it follows that $\F(t, -t, x_1, x_2, \ldots)= \F(x_1, x_2, \ldots)$, which proves that $F_{2n-1} \in \Gamma$.

For the general case, one can mimic the proof of \cite[Prop. 7.17.7]{Sta99} for the particular case of $\mu = \emptyset$. Namely, by the Murnaghan--Nakayama rule \cite[Theorem 7.17.3]{Sta99}, the coefficient of $p_{\alpha}$ in $s_{\delta_n / \mu}$ is equal to $\sum_{T} (-1)^{ht(T)}$ where $T$ runs over all border strip tableaux of shape $\delta_n / \mu$ and type $\alpha$, and $ht(T)$ is the height of $T$. If $\alpha$ has any even part, reorder the parts such that this even part is the last nonzero entry. But then, one can see that there is no border strip tableau of shape $\delta_n / \mu$ and content $\alpha$, since any border strip of shape $\delta_n / \mu$ must have odd size. Thus, the coefficient of $p_{\alpha}$ is zero. We conclude that $s_{\delta_n / \mu} \in \Gamma$.
\end{proof}

From the previous lemma, it follows that  $F_{2n-1}$ and, more generally, $s_{{\delta_n}/{\mu}}$ have expansions in terms of the Schur $P$-functions. The purpose of this paper is to clarify this expansion.

\section{Main result}\label{sec:thm}


\begin{definition}
 A (shifted) standard Young tableau is \emph{alternating} if every odd number is an ascent and every even number is a descent. 
\end{definition}

\begin{example}\label{ex:alternating}
 The following are the only two shifted standard Young tableaux of size $7$ which are alternating:
\[
 \ycenter \young(1246,:357) \quad , \quad \young(1246,:35,::7).
\]
\end{example}

The following is a special case of our main result, Theorem \ref{th:main}.

\begin{theorem}\label{th:main1}
The skew Schur function $s_{{\delta_n}/{\delta_{n-2}}}$ can be expressed as a nonnegative sum of Schur $P$-functions. We have
\[
 s_{\delta_n / \mu} = \sum_{U \in \textsf{AltShSYT}(2n-1)} P_{\sh(U)},
\]
where \textsf{AltShSYT}$(2n-1)$ is the set of shifted SYT of size $2n-1$ which are alternating.
\end{theorem}

\begin{example}
From Example \ref{ex:alternating} it follows that
\[
s_{\delta_4 / \delta_2} = P_{43} + P_{421}. 
\] 
Similarly, 
\[
s_{\delta_5 / \delta_3} = P_{54} + 2P_{531}+P_{432} 
\] 
because the shifted SYT of size $9$ which are alternating are:
\[
 \ycenter 
 \young(12468,:3579)\, ,  \quad
 \young(12468,:357,::9)\, , \quad
 \young(12468,:359,::7) \,, \quad
 \young(1246,:358,::79) \,.
\]

\end{example}

\begin{definition}
 For a skew shape $\lambda / \mu$ of size $n$, the \emph{standard filling} $T_{\lambda / \mu}$ is given by filling the shape with the entries $1,2,\ldots,n$, starting from the bottom row and moving up, subsequently filling each row from left to right. To distinguish it from the SYTs, we color it \blue{blue}.\footnote{If you printed this paper in black and white, you are not missing much.}
\end{definition}

\begin{example}
 For the shape $54321 / 32$, we have
\[
\ycenter T_{54321 / 32} = \blue{\young(:::9{\ten},::78,456,23,1)}.
\]
\end{example}

\begin{definition} \cite{Shi99}
 A (shifted or unshifted) SYT $U$ of size $|\lambda| - |\mu|$ is said to be \emph{$\lambda / \mu$--compatible} if
\begin{itemize}
\item whenever $T_{\delta_n/\mu}$ contains \blue{$\ycenter \young(i{\ip})$}, $i$ is a descent in $T$.
 \item whenever $T_{\delta_n/\mu}$ contains \blue{$\ycenter \young(j,i)$}, $(i,j)$ is an ascent in $T$.
\end{itemize}
\end{definition}

\begin{remark}\label{re:alternating}
Note that a (shifted) standard Young tableau is alternating if and only if it is $\delta_n / \delta_{n-2}$--compatible. 
\end{remark}

\begin{example}\label{ex:compatible}
 The following are the only two $54321 / 32$--compatible shifted standard Young tableaux:
\[
\ycenter \young(12479,:358{\ten},::6) \quad \text{and} \quad \young(12479,:358,::6{\ten})
\]
\end{example}

The following is our main result.

\begin{theorem}\label{th:main}
For any shape $\mu \subset \delta_n$, the skew Schur function $s_{\delta_n / \mu}$ can be expressed as a nonnegative linear combination of Schur $P$-functions. We have
\[
 s_{\delta_n / \mu} = \sum_{U \in \textsf{CompShSYT}(\delta_n / \mu)} P_{\sh(U)},
\]
where \textsf{CompShSYT}$(\delta_n / \mu)$ is the set of shifted SYT tableau which are $\delta_n / \mu$--compatible.
%
\end{theorem}

\begin{example}
In light of Example \ref{ex:compatible}, Theorem \ref{th:main} says that
 \[
  s_{54321 / 32} = P_{541} + P_{532}.
 \]
\end{example}

Note that Theorem \ref{th:main1} is a special case of Theorem \ref{th:main}, by Remark \ref{re:alternating}.

\section{Jeu de taquin and $\delta_n/\mu$--compatibility}\label{sec:jdt}

\begin{definition}
 Let $T$ be a SYT. Consider $T$ as a skew shifted SYT in the shifted plane. Denote by $\jdt(T)$ its \emph{shifted jeu de taquin rectification}, inspired by the notation and terminology in \cite[A1.2]{Sta99}.
\end{definition}

\begin{example}\label{ex:jdt} 
 \[
  \ycenter \jdt \left( \young(1247,359,6{\ten},8) \right) = \young(12479,:358,::6{\ten}).
 \]
Recall that in each step or \emph{slide} or jeu de taquin, we choose an empty internal corner, move the smaller of its (one or two) neighbors into this empty cell, then fill the resulting cell in the same way, and continue until we reach an external corner, and obtain a skew SYT. We do this subsequently until we obtain a shifted SYT, which turns out to be independent of the choices made \cite[A1.2]{Sta99}. For instance, we can compute the jeu de taquin rectification above as follows:
 \[
  \ycenter 
  \young(\,\,\,1247,:\,\,359,::\,6{\ten},:::8) \mapsto
  \young(\,\,\,1247,:\,\,359,::68{\ten}) \mapsto
  \young(\,\,\,1247,:\,359,::68{\ten}) \mapsto
  \young(\,\,1247,:\,359,::68{\ten}) 
 \]
 \[ 
 \ycenter
\mapsto  \young(\,\,1247,:3589,::6{\ten}) 
\mapsto  \young(\,1247,:3589,::6{\ten})
\mapsto  \young(12479,:358,::6{\ten})\,.
 \]
\end{example}

Our crucial technical lemma says that $\delta_n / \mu$--compatibility is well behaved  under jeu de taquin. This is not true for $\lambda / \mu$--compatibility in general: for
\[
\ycenter \jdt \left(\young(1,2)\right) = \young(12)\, ,
\]
jeu de taquin makes the tableau lose its $(2)/\emptyset$--compatibility and gain $(1,1)/\emptyset$--compatibility.

We first give a short argument for the special case of $\delta_n/\delta_{n-2}$, and then a different (and necessarily more intricate) argument for the general case.

%
%
%
%
%
%

\begin{proposition}\label{prop:compatible}
A standard Young tableau $T$ is alternating if and only if $\jdt(T)$ is alternating.
\end{proposition}

\begin{proof}
The \emph{reading word} $\mread(T)$ of a tableau $T$ is the word formed by subsequently reading each row from left to right, starting from the bottom row and moving up. 
Notice that $i$ is an ascent (descent) in $T$ if and only if it is an \emph{ascent} (\emph{descent}) in $\mread(T)$, in the sense that $i$ appears before (after) $i+1$ in the word.\footnote{This is sometimes called a \emph{right ascent} (\emph{right descent}) of the word.}

Now consider a skew shifted SYT $T$ and its shifted jeu de taquin rectification  $U=\jdt(T)$. By \cite[Theorem 7.1] {Sag87} and \cite[Theorem 6.10]{Hai89}, $\mread(T)$ and $\mread(U)$ are equivalent modulo the  Sagan--Worley relations \cite{Sag87}:
\begin{eqnarray*}
ab \cdots 	& \approx & ba \cdots \qquad \quad \quad \textrm{ for }  a<b, \\
 \cdots bac \cdots & \approx &  \cdots bca \cdots \qquad \,  \textrm{ for } a<b<c, \\
 \cdots cab \cdots & \approx &  \cdots acb \cdots \qquad  \, \textrm{ for } a<b<c ,
\end{eqnarray*}
where the letters represented by $\cdots$ remain the same.

We now prove that jeu de taquin preserves peaks, by proving that  $\mread(T)$ and $\mread(U)$ have the same peaks. 
Since the Sagan-Worley moves are reversible, we only need to check that a move cannot turn a peak $i$ into a non-peak. This follows from the following observation: $i$ is a peak in a permutation if and only if it is preceded by both $i-1$ and $i+1$. This property cannot be changed by any of the Sagan-Worley relations: The first relation cannot involve $i$, and the second and third  can never change the relative order of two consecutive numbers.

Finally notice that a tableau of size $2n-1$ is alternating if and only if its set of peaks is $\{2, 4, \ldots, 2n-2\}$. This property is preserved by jeu de taquin rectification.
\end{proof}

The previous proof relies heavily on the description of alternating tableau in terms of  peaks; notice that the Sagan-Worley relations do not respect the ascents and descents. We do not know how to extend this argument to the setting of  $\delta_n/\mu$--compatibility. To settle this general case, we will carry out a careful analysis of the jeu de taquin algorithm from the point of view of $\delta_n/\mu$--compatibility.
We will keep referring back to the following:

\medskip

\begin{tabular}{|p{10cm}|}
\hline
\smallskip

For a $\delta_n/\mu$--compatible tableau $T$:

\medskip

\noindent $\bullet$ If $T_{\delta_n/\mu}$ contains \blue{$\ycenter \young(j,i)$}\,, then $(i,j)$ is an \textbf{ascent}  in $T$:

\medskip
\qquad 
$\begin{array}{c}
\begin{picture}(40,40)(0,0)
\put(10,10){\makebox{$i$}}
\put(35,25){\makebox{$j$}}
\put(1,32){\line(1,0){30}}
\put(31,2){\line(0,1){30}}
\end{picture} \\
\mbox{$i$ is southWest of $j$,}  
\end{array}$
\qquad \qquad   
$\begin{array}{c}
\begin{picture}(40,40)(0,0)
\put(0,0){\makebox{$i$}}
\put(20,20){\makebox{$j$}}
\put(7,0){\line(1,0){30}}
\put(7,0){\line(0,1){30}}
\end{picture} \\
\mbox{$j$ is northEast of $i$.}  
\end{array}$

\bigskip

\noindent $\bullet$
If $T_{\delta_n/\mu}$ contains \blue{$\ycenter \young(i{\ip})$}\,, then $(i,i+1)$ is a \textbf{descent}  in $T$:
\medskip

\qquad 
$\begin{array}{c}
\begin{picture}(40,40)(0,0)
\put(0,0){\makebox{$i+1$}}
\put(20,20){\makebox{$i$}}
\put(0,10){\line(1,0){35}}
\put(0,10){\line(0,1){30}}
\end{picture} \\
\mbox{$i$ is Northeast of $i+1$,}  
\end{array}$
\quad \qquad   
$\begin{array}{c}
\begin{picture}(40,40)(5,0)
\put(10,10){\makebox{$i+1$}}
\put(35,25){\makebox{$i$}}
\put(0,22){\line(1,0){40}}
\put(40,0){\line(0,1){22}}
\end{picture} \\
\mbox{$i+1$ is Southwest of $i$.}  
\end{array}$

\\
\hline
\end{tabular}

\bigskip

\begin{proposition}\label{prop:compatible}
A standard Young tableau $T$ is $\delta_n / \mu$--compatible if and only if $\jdt(T)$ is $\delta_n / \mu$--compatible.
\end{proposition} 

\begin{proof}
During the procedure of  jeu de taquin rectification, we call the move (up or left) of a single number a \emph{move}, and a series of (upward and leftward) moves transforming an inner corner into an outer corner a \emph{slide}. For instance
\[
\ycenter   \young(\,\,1247,:\,359,::68{\ten}) \mapsto  \young(\,\,1247,:3589,::6{\ten}) 
\]
is a slide consisting of four moves. We will prove that a slide cannot affect the $\delta_n/\mu$--compatibility of a skew shifted SYT, which will show the desired result.

\medskip

For the sake of contradiction, assume that a slide of jeu de taquin, which transformed a tableau $T_1$ into a tableau $T_2$, affected  $\delta_n / \mu$-compatibility. There are two (not mutually exclusive) cases, namely:
\begin{itemize}
 \item the tableau gained/lost  an ascent $(i,j)$ prescribed by   \blue{$\ycenter \young(j,i)$} in $T_{\delta_n / \mu}$, or
 \item  the tableau gained/lost a descent $(i, i+1)$ prescribed by  \blue{$\ycenter \young(i{\ip})$} in $T_{\delta_n / \mu}$.
\end{itemize}
We will study these two cases separately.

\medskip
%
%
%
%
%
%
%
%
%
%
%
%
%

\noindent \textbf{Case 1:} The tableau gained or lost an ascent $(i,j)$ prescribed by  $T_{\delta_n / \mu}$, with $i<j$. 

\smallskip

Assume that $i$ is minimal among all such ascents. 
We consider four subcases, namely when $i$ remains still and $j$ moves left or up, and when $i$ moves left or up.

\smallskip

\noindent \textbf{Case 1.1:} During the slide, $i$ did not move and $j$ moved left.

\smallskip

The area southWest of $j$ before the move of $j$ contains the area southWest of $j$ after the move, so the tableau must have lost the descent: 
Before the move $i$ was southWest of $j$, and after the move it is not, making the tableau $T_1$ lose its $\delta_n/\mu$-compatibility when it turned into $T_2$. Since $i$ did not move, it must have been on the column directly left of $j$'s column, and below $j$. But then the move put $j$ above $i$ in $T_2$, a contradiction.

\smallskip

\noindent \textbf{Case 1.2:} During the slide, $i$ did not move and $j$ moved up.

\smallskip

Before the move $i$ was not southWest of $j$, and after the move it is. The tableau gained a prescribed ascent, making $T_2$ a $\delta_n/\mu$--compatible tableau.

Since $i$ did not move, it must have been on the row directly above $j$'s row, and strictly left of $j$. In $T_2$, $i$ and $j$ are on the same row. If there was a number $x$ between them, it would satisfy $i<x<j$. 
In $T_{\delta_n/\mu}$, because $i$ is directly below $j$, $x$ would have to be either directly east of $i$ (and therefore Southwest of $i$ in $T_2$) or directly west of $j$ (and therefore Northeast of $j$ in $T_2$) -- a contradiction in either case. It follows that the slide looked like this:
\[
\ycenter
\young(:\,,iy,:j) \quad  \mapsto \quad
\young(:y,i\,,:j)
 \mapsto \quad
\young(:y,ij,:\,)
\]
where the number $y$ must move up since $i$ does not move. We have $i<y<j$ which, by the argument in the previous paragraph, means that 
\[
\ycenter
T_{\delta_n/\mu} \quad \textrm{ contains } \blue{\quad \young(y{\cdots}j,\,\,i)}
\]
In $T_2$, $j$ is Southwest of $j-1$, which is Southwest of $j-2$, , $\cdots$, which is Southwest of $y$. But $y$ and $j$ are adjacent, so $y=j-1$. Now
\[
\ycenter
T_{\delta_n/\mu} \quad \textrm{ contains } \blue{\quad \young({\jm}j,{\im}i)}
\]
which forces $i-1$ to be Northeast of $i$ and southWest of $j-1$ in $T_2$; \emph{i.e.}, directly above $i$. Therefore the slide looked like:
\[
\ycenter
\young({\im}z,i{\jm},:j) \quad  \mapsto \quad
\young({\im}{\jm},i\,,:j) \quad  \mapsto \quad
\young({\im}{\jm},ij,:\,)
\]
where the number $z$ must have moved up, or else it would be between $i-1$ and $i$. 

We conclude that this slide also made the tableau gain the (smaller) ascent $(i-1,j-1)$, while leaving $i-1$ still and moving $j-1$ up. This contradicts the minimality of $i$.

\smallskip

\noindent \textbf{Case 1.3:} During the slide, $i$ moved left.

\smallskip

Here we gained the prescribed ascent $(i,j)$, making $T_2$  $\delta_n/\mu$--compatible.

In $T_2$, $j$ must be on the column to the right of $i$'s column. It cannot be higher than $i$, or else it would have been on the same column and above $i$ in $T_1$. Therefore it must be directly to the right of $i$, having slid into $i$'s old position. Since $j$ was not northEast of $i$, it must have slid up from below $i$, so the slide looked like this:
\[
\ycenter
\young(\,i,:j) \quad  \mapsto \quad
\young(ij,:\,)
\]
We need to consider two subcases.

\noindent \textbf{Case 1.3.1:} There is no cell to the left of $j$ in $T_1$. From the shape of $\delta_n$, $j+1$ is to the right of $j$ in $T_{\delta_n/\mu}$, so it must be Southwest of $j$ in $T_2$. The only possibility is that it is directly below $j$, and was to the right of $j$ in $T_1$. The slide must have looked like:
\[
\ycenter
\young(\,ix,:j{\jp}) \quad  \mapsto \quad
\young(ijx,:{\jp})
\]
where $j<x<j+1$, a contradiction.

\noindent \textbf{Case 1.3.2:} There is a cell to the left of $j$ in $T_1$. The number in it must be between $i$ and $j$, and by the same argument of Case 1.2, it must actually equal $i+1$, and
\[
\ycenter
T_{\delta_n/\mu} \quad \textrm{ contains } \quad \blue{\young(j{\jp},i{\ip})}
\]
By $\delta_n/\mu$-compatibility, $j+1$ must be directly below $j$ and to the right of $i+1$ in $T_2$, making the slide look like:
\[
\ycenter
\young(\,i,{\ip}j) \quad  \mapsto \quad
\young(ij,{\ip}{\jp})
\]
Note that $j+1$ could not be to the right of $j$ in $T_1$, or else the number directly above it would have to be between $j$ and $j+1$. So $j+1$ must have been below $j$ and slid up:
\[
\ycenter
\young(\,i,{\ip}j,:{\jp}) \quad  \mapsto \quad
\young(ij,{\ip}{\jp})\,.
\]
Therefore the slide introduced the ascent $(i+1,j+1)$ stipulated by $T_{\delta_n/\mu}$, leaving $i+1$ still and moving $j+1$. As we saw in Case 1.2, this is impossible.

\smallskip

\noindent \textbf{Case 1.4:} During the slide, $i$ moved up.
\smallskip

Here we lost the ascent $(i,j)$ when we go from the $\delta_n/\mu$--compatible tableau $T_1$ to $T_2$. Then $j$ must be on the same row as $i$ in $T_1$; arguing as above, it must actually be directly to the right of $j$. The number $x$ directly above $j$ must have stayed still, so the slide looks like:
\[
\ycenter
\young(\,x,ij) \quad  \mapsto \quad
\young(ix,??)\,.
\]
where $j$ may or may not have moved left, so we do not specify the bottom row in $T_2$. As in the previous cases, $i<x<j$ implies that $x=j-1$ and that 
\[
\ycenter
T_{\delta_n/\mu} \quad \textrm{ contains } \quad  \blue{\young({\jm}j,{\im}i)}\, ,
\]
and $\delta_n/\mu$-compatibility then gives that the slide was
\[
\ycenter
\young({\im}{\jm},ij) \quad  \mapsto \quad
\young(i{\jm},??)\,.
\]
If $i-1$ slid up, then the prescribed ascent $(i-1,j-1)$ would also be lost by moving $i-1$ up, contradicting the minimality of $i$. Therefore the slide was
\[
\ycenter
\young(\,{\im}{\jm},:ij) \quad  \mapsto \quad
\young({\im}i{\jm},:??)\,.
\]
If there was a cell to the left of $i$ in $T_1$, the number in it would need to be between $i-1$ and $i$; so this move went along the bottom left diagonal of the board. Also, 
\[
\ycenter
T_{\delta_n/\mu} \quad \textrm{ contains } \quad \blue{\young({\jm}j,{\im}i,h)}
\]
for some $h$. By $\delta_n/\mu$-compatibility, $h$ must have been to the left of $i-1$ in $T_1$. Because we are at the bottom of the board, the slide must have looked like this:
\[
\ycenter
\young(h{\im}{\jm},:ij) \quad  \mapsto \quad
\young(h,{\im}i{\jm},:??)\,.
\]
But then the prescribed ascent $(h,i-1)$ was lost by moving $h$ up, contradicting the minimality of $i$.

\bigskip

\noindent \textbf{Case 2:} The tableau gained or lost a descent $(i,i+1)$ prescribed by  $T_{\delta_n / \mu}$.

\smallskip

Again, we consider the same four subcases as above:

\smallskip

\noindent \textbf{Case 2.1:} During the slide, $i$ did not move and $i+1$ moved up.

\smallskip

Before the move $i$ was Northeast of $i+1$, and after the move it is not. Since $i$ did not move, it must have been on the row directly above $i+1$ and east of it. Therefore the move placed $i+1$ on the same row and to the left of $i$, a contradiction.

\smallskip

\noindent \textbf{Case 2.2:} During the slide, $i$ did not move and $i+1$ moved left.

\smallskip

In this case the tableau must have gained the descent: Before the move $i$ was not Northeast of $i+1$, and after the move it is. Since $i$ did not move, it is on the same column and (necessarily directly) above $i+1$ after the move. The slide must have looked like this:
\[
\ycenter
\young(:i,\,x{\ip}) \quad  \mapsto \quad
\young(:i,x\,{\ip}) \quad \mapsto \quad
\young(:i,x{\ip}\,)  
\]
which gives $i<x<i+1$, a contradiction. 


\pagebreak

\noindent \textbf{Case 2.3:} During the slide, $i$ moved up.
\smallskip

Before the move, $i+1$ was not Southwest of $i$, and after the move it is. Since $i+1$ could not have been on the same row and to the left of $i$ before, it must have been directly to the right of $i$, and must have slid into $i$'s old position:
\[
\ycenter
\young(\,,i{\ip}) \quad  \mapsto \quad
\young(i,\,{\ip}) \quad \mapsto \quad
\young(i,{\ip}\,)\,.  
\]
But then the cell northeast of these is part of the tableau:
\[
\ycenter
\young(\,x,i{\ip}) \quad  \mapsto \quad
\young(ix,\,{\ip}) \quad \mapsto \quad
\young(ix,{\ip}\,)\,,
\]
and the number $x$ in it satisfies $i<x<i+1$, a contradiction.

\smallskip

\noindent \textbf{Case 2.4:} During the slide, $i$ moved left.

\smallskip

Before the move, $i+1$ was Southwest of $i$, and after the move it is not. Therefore $i+1$ must have been on the same column as $i$ and (necessarily directly) below it. The slide must have looked like this:
\[
\ycenter
\young(\,i,*{\ip}) \quad  \mapsto \quad
\young(i{\ip},*)
\]
The tableaux cannot contain the cell with the asterisk, because the number in it would need to be between $i$ and $i+1$. Therefore this part of the slide is happening along the lower diagonal of the tableaux.

Because $i+1$ is not a descent in $T_1$, it must be the rightmost number in its row in $T_{\delta_n/\mu}$. Given the shape of $\delta_n$,
\[
\ycenter
T_{\delta_n/\mu} \quad \textrm{ contains } \quad \blue{\young(i{\ip},{\jp})}\, .
\]
for some $j+1<i$. Then $(j+1,i)$ must be an ascent in $T_1$, which implies that the slide looked like this

\[
\ycenter
\young(\,,{\jp}i,:{\ip}) \quad  \mapsto \quad
\young({\jp},i{\ip},:\,)\,.
\]
This means that this slide made the tableau lose the prescribed ascent $(j+1,i)$ which, as we saw in Case 1, leads to a contradiction.
\end{proof}

\section{Proof of Theorem \ref{th:main}}\label{sec:proof}

We will need the following two theorems on the Schur expansions of skew Schur functions and Schur $P$-functions.

\begin{theorem}[Shimozono \cite{Shi99}]\label{th:shimozono} 
We have
 \[
  s_{\lambda / \mu} = \sum_{\nu} c^{\lambda}_{\mu,\nu} s_{\nu},
 \]
where the Littlewood--Richardson coefficient $c^{\lambda}_{\mu,\nu}$ is equal to the number of SYT $T$ of shape $\nu$ which are $\lambda / \mu$--compatible.
\end{theorem}

\begin{proof}
 The proof is  in \cite{Shi99}, but note that Shimozono's definition of $\lambda / \mu$-compatibility differs by ours in the sense that it reverses ascents and descents. The change is easily made by considering the reverse alphabet $\cdots > 3 > 2 > 1$.
\end{proof}


\begin{theorem}[Stembridge \cite{Ste89}]\label{th:stembridge} 
Fix a shifted SYT $U$ of shape $\lambda$. We have
 \[
  P_{\lambda} = \sum_{\mu} g^{\lambda}_{\mu} s_{\mu},
 \]
where $g^{\lambda}_{\mu}$ is the number of SYT $T$ of shape $\mu$ such that $\jdt(T) = U$.
\end{theorem}

We have now assembled all the ingredients to prove the main theorem.

\begin{proof}[Proof of Theorem \ref{th:main}]
Denote the set of shifted standard Young tableaux of shape $\delta_n / \mu$ by \textsf{ShSYT}
, and the set of (shifted) standard $\delta_n / \mu$--compatible tableaux by \textsf{CompSYT}
(\textsf{CompShSYT}
). By Theorem \ref{th:shimozono} we have
\begin{eqnarray*}
 s_{\delta_n / \mu}	& = & \sum_{T \in \textsf{CompSYT}}
 s_{\sh(T)} \\
& = & \sum_{U \in \textsf{ShSYT}}\,\,  \sum_{T \in \textsf{CompSYT}
\, : \,  U = \jdt(T)} s_{\sh(T)}.
\end{eqnarray*}
By Proposition \ref{prop:compatible} and Theorem \ref{th:stembridge} respectively, this equals
\begin{eqnarray*}
 s_{\delta_n / \mu}
&=& \sum_{U \in \textsf{CompShSYT}}
\,\, \sum_{T \in \textsf{SYT} \, : \, U = \jdt(T)}  s_{\sh(T)}  \\
				& = & \sum_{U \in \textsf{CompShSYT}}
				P_{\sh(U)}
\end{eqnarray*}
as we wished to prove. \end{proof}

\section{Further Work}

\begin{itemize} 
\item
As mentioned earlier, the Schur and the Schur $P$-functions are related to the representations and the projective representations of the symmetric group, and to the cohomology of the Grassmannian and the isotropic Grassmannian. 
The representation theoretic and geometric significance of Theorem \ref{th:main} should be explored.
\item
Theorem \ref{th:main} implies that if $\lambda / \mu$ is a disjoint union of staircase skew shapes and their $180$ degree rotations, then $s_{\lambda / \mu}$ is Schur $P$-positive. It is natural to wonder whether these are the only skew Schur functions which are Schur $P$-positive. In fact, Dewitt \cite{Dew12} has proved the stronger statement that these are the only skew Schur functions which are \emph{linear} combinations of Schur P-functions.
\end{itemize}

\end{document}